\documentclass{article}

\usepackage[utf8]{inputenc}
\usepackage[english]{babel}
\usepackage[T1]{fontenc}
\usepackage{color}
\usepackage{amsfonts}
\usepackage{amsmath}
\usepackage{amssymb}
\usepackage{amstext}
\usepackage{amsthm}
\usepackage{mathrsfs}
\usepackage{amsmath,amscd}
\usepackage{cite}
\usepackage{enumitem}
\usepackage{mathtools}

\newtheorem{theorem}{Theorem}[section]
\newtheorem{corollary}[theorem]{Corollary}
\newtheorem{lemma}[theorem]{Lemma}
\newtheorem{proposition}[theorem]{Proposition}

\newtheorem{conjecture}{Conjecture}[section]

\theoremstyle{definition}

\newtheorem{remark}[theorem]{Remark}

\newcommand{\F}{\mathbb{F}}

\newcommand{\Q}{\mathbb{Q}}

\newcommand{\Z}{\mathbb{Z}}

\newcommand{\p}{\mathfrak{p}}

\newcommand{\q}{\mathfrak{q}}

\newcommand{\eps}{\varepsilon}

\newcommand{\abs}[1]{\left\lvert#1\right\rvert}

\newcommand{\ord}{\text{ord}}

\title{Primitive divisors of sequences associated to elliptic curves}
\author{Matteo Verzobio}
\date{}

\begin{document}
	\maketitle
	
	\begin{abstract}
		Let $\{nP+Q\}_{n\geq0}$ be a sequence of points on an elliptic curve defined over a number field $K$. In this paper, we study the denominators of the $x$-coordinates of this sequence. We prove that, if $Q$ is a torsion point of prime order, then for $n$ large enough there always exists a primitive divisor. Later on, we show the link between the study of the primitive divisors and the Lang-Trotter conjecture. 
	\end{abstract}

	\section{Introduction}
	
	Let $E$ be an elliptic curve defined over a number field $K$, $P$ a non-torsion point of $E(K)$ and $Q$ a point of $E(K)$ such that, for every $n>0$, $nP\neq -Q$. Since every fractional ideal of $K$ has a unique factorization, we can write
	\[
	(x(nP+Q))=\frac{C_n}{D_n}
	\]
	where $C_n$ and $D_n$ are two relatively prime integral ideals. If $K=\Q$, then $D_n$ can be represented uniquely by a positive integer. We want to understand when a term of the sequence $\{D_n\}_{n>0}$ has a primitive divisor, i.e., when there exists a prime ideal $\p$ such that \[\p \nmid D_1D_2\cdots D_{n-1} \mbox{    but    } \p\mid D_n.\]
	 If $\p$ is a primitive divisor of $D_n$, then $n$ is the smallest positive integer such that
	 \[
	 nP+Q\equiv O \mod{\p}.
	 \]  There are some results under the hypothesis $Q=O$. In 1988, Silverman proved the following theorem.
	\begin{theorem}[Silverman\cite{silverman}]\label{silv}
If $Q=O$ and $K=\Q$, then $D_n$ has a primitive divisor for every $n$ large enough.
	\end{theorem}
This result was generalized by Cheon and Hahn.
\begin{theorem}[Cheon, Hahn\cite{ChHa}]
	If $Q=O$ and $K$ is a number field, then $D_n$ has a primitive divisor for every $n$ large enough.
\end{theorem}
In the case when $Q\neq O$ it was proved, in \cite{shparlinski}, that \[\omega_K(\prod_{n=M}^{N+M}D_n)\gg N\] where $\omega_K$ is the function that counts the number of distinct prime divisors of an integral ideals in $K$. Thanks to this result it is reasonable to think that the result of Cheon and Hahn it is true also in the case when $Q\neq O$. We will generalize this result under the assumption that $Q$ is a torsion point of prime order.

\begin{theorem}\label{Teo1}
	Let $E$ be an elliptic curve defined over a number field $K$, let $P$ be a non-torsion point and $Q$ be a torsion point of prime order. Then there exists a constant $C$, such that, if $n>C$, then $D_n$ has a primitive divisor.
\end{theorem}

Later on, we will show why the study of the primitive divisors is related with the elliptic analogue of Artin's conjecture, the so called Lang-Trotter conjecture.

\begin{conjecture}[Lang-Trotter]
	Let $E$ be an elliptic curve defined over $\Q$ and $P$ be a point on $E(\Q)$. Let
	\[
	N_P(x)=\{p\leq x \text{  such that  }  p\nmid\Delta \text{  and  } \overline{P} \text{  generates  } E(\F_p)\},
	\]
	where $\overline{P}$ is the reduction of $P$ modulo $p$.
	Then,
	\[
	\#N_P(x) \sim \delta(P)\frac{x}{\log x}
	\]
	as $x$ goes to infinity. 
\end{conjecture}

We will study a set that it is related to $N_P(x)$. Let $P$ and $Q$ be two points in $E(\Q)$ and define
\[
N_{P,Q}(x):=\{p\leq x\text{  such that  }  p\nmid\Delta \text{  and  } \overline{Q} \in \langle \overline{P} \rangle\ \text{ in } E(\F_p)\},
\]
It is clear that
\[
N_{P,Q}(x) \subseteq N_P(x)
\]
since if $P$ generates $E(\F_p)$, then $Q$ is in the orbit of $P$ modulo $p$.
\begin{theorem}[\cite{seguinphd}]
The set $N_{P,Q}(x)$ is infinite as $x$ goes to infinite if $P$ and $Q$ have infinite order. Furthermore, if $E$ has rank $1$, then
\[
\#N_{P,Q}(x)\gg \sqrt{\log x}.
\]
\end{theorem}

We will generalize this result by showing that the condition that $E$ has rank $1$ and the condition that $Q$ has infinite order are unnecessary. 
\begin{theorem}\label{Teo2}
	Let $E$ be an elliptic curve defined over $\Q$ with $P$ a non-torsion point and $Q$ in $E(\Q)$. Then,
	\[
	\#N_{P,Q}(x)\gg \sqrt{\log x}.
	\]
\end{theorem}

In Section \ref{prelimin} we introduce some basic facts on the elliptic curves that we will use in the paper. Then, in Section \ref{secteo1}, we will prove Theorem \ref{Teo1}, using the same techniques introduced by Silverman in the proof of Theorem \ref{silv}. Finally, in Section \ref{secteo2}, we will prove Theorem \ref{Teo2} and we will show the relation between the study of the primitive divisors of $D_n$ and $N_{P,Q}(x)$.

	\section{Preliminaries on elliptic curves}\label{prelimin}
 
 Let $E$ be an elliptic curve defined by the equation
 \[
 y^2+a_1xy+a_3y=x^3+a_2x^2+a_4x+a_6,
 \]
 take $P$ a non-torsion point and $Q$ a torsion point.
 Let us define
	\[
	(x(nP))=\frac{A_n}{B_n}
	\]
	with $A_n$ and $B_n$ two relatively prime integral ideals.
	Given a valuation $\nu$ associated to a prime $\p\in K$, define
	\[
	h_\nu(P):=\text{max}\{0,-\nu(x(P))\}
	\]
	and
	\[
	h(P)=\frac{1}{[K:\Q]}\sum_{\nu \in M_K} n_\nu h_\nu(P),
	\]
	where $n_\nu$ is the degree of the local extension $K_\p/\Q_p$ and $M_K$ is the set of all the places of $K$. Given a point $P$ in $E(K)$, we define the canonical height as in \cite[Proposition VIII.9.1]{arithmetic}, i.e.
	\[
	\hat{h}(P)=\frac{1}{2}\lim_{N\to \infty}4^{-N}h(2^NP).
	\] First of all, we recall the properties of the height and of the canonical height that will be necessary in this paper. For details see \cite[Chapter 8]{arithmetic}.
	\begin{itemize}
		\item Given a point $R$, there exists a constant $C_R>0$ such that, for every $S$ in $E(K)$,
		\[
		h(R+S)\leq C_R+2h(S).
		\]
		\item There exists a constant $C_E$ such that, for every $R\in E(K)$,
		\[
		\abs{h(R)-2\hat{h}(R)}\leq C_E.
		\]
		\item The canonical height is quadratic, i.e.
		\[
		\hat{h}(nR)=n^2\hat{h}(R)
		\]
		for every $R$ in $E(K)$.
	\end{itemize}

	We will need also the following proposition.
	\begin{proposition}\label{Siegel}
		Given an absolute value $\nu$, for every $\eps>0$ there exists an $N=N(\eps,\nu,P,Q)$ such that, for every $n>N$,
		\[
		h_\nu(nP+Q)\leq \eps h(nP+Q).
		\]
	\end{proposition}
	\begin{proof}
			Define $P_n:=nP+Q$. Then,
		\[
		h(P_n)\geq \frac{h(P_n-Q)-C_{-Q}}{2}\geq \hat{h}(nP)-\frac{C_E}{2}-\frac{C_{-Q}}{2}= n^2\hat{h}(P)-\frac{C_E}{2}-\frac{C_{-Q}}{2}
		\]
		and so $h(P_n)\to \infty$ as $n$ goes to infinity.
		We conclude using Siegel's Theorem (see \cite[Theorem IX.3.1]{arithmetic}), since
		\[
		\lim_{h(P_i)\to \infty}\frac{h_v(P_i)}{h(P_i)}=0.
		\]
	\end{proof}

Let $S$ be the finite set of places of $K$ composed by
\begin{itemize}
	\item all archimedean places;
	\item all places over primes where $E$ has bad reduction and the primes $\p$ where the equation defining the elliptic curve is not minimal in $K_\p$;
	\item all places over the primes ramified in $K$;
	\item all places over primes dividing $2\cdot\ord(Q)$, where $\ord(Q)$ is the order of $Q$ in $E$;
	\item the places over the primes where $Q$ reduces to the identity.
\end{itemize}

\begin{lemma}\label{formal}
If $\nu$ is a place not in $S$ and $M$ is a point of $E(K)$ such that $h_\nu(M)>0$, then
	\[
	h_\nu(nM)=h_\nu(M)+2\nu(n).
	\]
\end{lemma}

\begin{proof}
	This follows from the properties of the formal group of the elliptic curves. For details see \cite[Lemma page 200]{ChHa}.
\end{proof}

\begin{lemma}\label{quad}
	Given $R$ and $M$ in $E(K)$, there exists a constant $C(R,M)>0$ such that
	\[
	\hat{h}(nR+M)\geq\hat{h}(nR)-nC(R,M),
	\]
	for every $n\geq 1$.
\end{lemma}

\begin{proof}
	The form $\langle R,M\rangle =\hat{h}(R+M)-\hat{h}(R)-\hat{h}(M)$ is bilinear and thus
	\begin{align*}
	0&=\langle nR,M\rangle -n\langle R,M\rangle \\&=\hat{h}(nR+M)-\hat{h}(nR)-\hat{h}(M)+n\Big(-\hat{h}(R+M)+\hat{h}(R)+\hat{h}(M)\Big).
	\end{align*}
It follows that,
\[
\hat{h}(nR)=\hat{h}(nR+M)+n\Big(-\hat{h}(R+M)+\hat{h}(R)+\hat{h}(M)-\frac{\hat{h}(M)}{n}\Big)
\]
and so we can choose
\[
C(R,M)=\hat{h}(R)+\hat{h}(M).
\]
\end{proof}

\begin{lemma}\label{ord}
	Let $R$ and $M$ be two points of $E$, $\nu$ an absolute value not in $S$ and suppose $h_\nu(R)>0$ and $h_\nu(M)>0$. Then
	\[
	h_\nu(R+M)\geq \min\{h_\nu(R),h_\nu(M)\}.
	\]
\end{lemma}
\begin{proof}
	We use the formal group as defined in \cite{arithmetic}. For every $P$, define 
	\[
	z(P):=\frac{x(P)}{y(P)}
	\]
	and so \[\ord_\p(z(P))=\frac{h_\nu(P)}{2},\]where $\p$ is the prime associated to $\nu$. Indeed, using the equation generating the elliptic curve and that $\nu$ is not in $S$ we have
	\[
	3\ord_\p(x(P))=2\ord_\p(y(P))<0.
	\] Let $j:=\min\{h_\nu(R),h_\nu(M)\}$. Therefore,
	\[
	z(R+M)\equiv z(R)+z(M)\equiv 0\mod{\p^{\frac{j}{2}}}.
	\]
	thanks to \cite[Proposition 7.2.2.]{arithmetic}.
	Thus, 
	\[\frac{h_\nu(R+M)}{2}=\ord_\p(z(R+M))\geq\frac{j}{2}\]
	and this concludes the proof.
\end{proof}

\section{Proof of Theorem \ref{Teo1}}\label{secteo1}

We are now ready to give the proof of Theorem \ref{Teo1}. 
We will use two trivial estimates. Let $\mu(n)$ the number of divisors of $n$. Then,
\[
\mu(n)\leq n.
\]
Given $x\in K$, define
\[
h(x):=\frac{1}{[K:\Q]}\sum_\nu n_\nu h_\nu(x)
\] 
where the sum runs over all the places of $K$ and $h_\nu(x)=\max\{0,-\nu(x)\}$. It is well known that $h(d)=\log d$ if $d$ is in $\Z$ and thus
\[
\sum_{d|n}h(d)=\sum_{d|n}\log d\leq \mu(n)\log n\leq n\log n.
\]

\begin{proof}[Proof of Theorem \ref{Teo1}]

Suppose that $Q$ is a $r$-torsion point, with $r$ a prime. Define
\[
K:=\max_{0\leq m\leq r-1} C_{mQ}
\]
with $C_R$ as defined in the previous section.
Hence, for every $n$ and for every $M\in E(K)$,
\[
h(M+nQ)\leq 2h(M)+K.
\]
		Take $n>1$ and suppose that $D_n$ has not a primitive divisor. We will prove that this is not possible for $n$ large enough. So, for every prime $\q$ over a valuation not in $S$, if $\q$ divides $D_n$, then it divides $D_{n-k}$ for some $k>0$ and therefore
		\[
		kP\equiv nP+Q-(n-k)P-Q\equiv O-O\equiv O \mod{\q}.
		\]
		Here $O$ is the identity of the elliptic curve and, given two point $R$ and $M$ in $E(K)$, we say that $R\equiv M\text{ mod }\p$ if the reduction of $R$ is equal to the reduction of $M$ modulo $\p$.
		Hence,
		\[
		rnP\equiv r(nP+Q)\equiv rO\equiv O \mod{\q}.
		\]
		Let $j:=\gcd(rn,k)$ and by Bézout's identity there exist $a$ and $b$ such that \[rna+bk=j.\] It follows that
		\[
		jP\equiv rnaP+bkP\equiv aO+bO\equiv O\mod{\q}
		\]
		and so $\q$ divides $B_{j}$ for some $j<n$ that divides $rn$. If $j$ divides $n$, then
		\[
		Q\equiv nP+Q\equiv O\mod{\q} 
		\]
		and this is absurd since $\q$ is not associated to a valuation in $S$. Then, for every divisor $\q$ of $D_n$, $\q$ divides $B_{rn/d}$ with $d$  a divisor of $n$, coprime with $r$ and greater than $r$.
	\begin{lemma}
		If $\nu$ is a valuation not in $S$ associated to a prime $\q$, that divide $D_n$ and $B_{rn/d}$ with $(r,d)=1$, then there exists a point $Q_d$, multiple of $Q$, such that
		\[
		h_\nu(nP+Q)\leq h_\nu(\frac{n}{d}P+Q_d)+2h_\nu(d).
		\]
	\end{lemma}
\begin{proof}
 Define \[P_1=\frac{rn}{d}P \mbox{  and  } P_2=nP+Q.\] Thus,\[P_1\equiv P_2\equiv O \mod{\q}\]
 and
	\[
	h_\nu(P_2)=h_\nu(rP_2)=h_\nu(rnP)=h_\nu(\frac{rn}{d}P)+2h_\nu(d)=h_\nu(P_1)+2h_\nu(d).
	\]
	Since $d$ is coprime with $r$ we have $(r,d-r)=1$ and therefore
	\[
	ar+b(d-r)=1
	\]
	with $a$ and $b$ integers.
	Then,
	\begin{align*}
	\frac{n}{d}P+bQ&=a\frac{rn}{d}P+b\frac{n(d-r)}{d}P+bQ\\&=a\frac{rn}{d}P+b(nP+Q-\frac{rn}{d}P)\\&=(a-b)P_1+bP_2.
	\end{align*}
	So, thanks to Lemma \ref{ord},
	\[
	h_\nu(\frac{n}{d}P+bQ)\geq h_\nu(P_1)=h_\nu(P_2)-2h_\nu(d).
	\]
	Hence, define $Q_d=bQ$ and this concludes the proof.
\end{proof}

Observe that $Q_d$ does not depend on $\nu$, but only on $Q$ and $d$. If $\q$ is the prime associated to $\nu$ and $h_\nu(nP+Q)>0$, then $\q$ divides $D_n$ and so there exists $d$ such that $\q$ divides $B_{nr/d}$, $(r,d)=1$ and $d>r$. Thus, we can apply the previous lemma, obtaining that, if $h_\nu(nP+Q)>0$, then there exists $d$ such that \[h_\nu(nP+Q)\leq h_\nu(\frac{n}{d}P+Q_d)+2h_\nu(d).\] Therefore,
\[
n_\nu h_\nu(nP+Q)\leq \sum_{\substack{d|n \\ d>r}} n_\nu h_\nu(\frac{n}{d}P+Q_d)+2n_\nu h_\nu(d)
\]
since every addend in the RHS is greater than $0$ and the LHS is less than an addend in the RHS. Hence,
\[
\sum_{\nu \not\in S}n_\nu h_\nu(nP+Q)\leq \sum_{\nu \not\in S}\sum_{\substack{d|n \\ d>r}} n_\nu h_\nu(\frac{n}{d}P+Q_d)+2n_\nu h_\nu(d).
\]


So,
\begin{align*}
\frac{1}{[K:\Q]}\sum_{\nu \not\in S}n_\nu h_\nu(nP+Q)&
\leq\frac{1}{[K:\Q]} \sum_{\nu \not\in S}\sum_{\substack{d|n \\ d>r}}n_\nu h_\nu(\frac{n}{d}P+Q_d)+2n_\nu h_\nu(d)\\&
\leq \sum_{\substack{d|n \\ d>r}}h(\frac{n}{d}P+Q_d)+2h(d)\\&
\leq \sum_{\substack{d|n \\ d>r}}2h(\frac{n}{d}P)+2h(d)+K\\&
\leq \Big(\sum_{\substack{d|n \\ d>r}}4\hat{h}(\frac{n}{d}P)\Big)+(K+2C_E)\mu(n)+2n\log n\\&
\leq 4\hat{h}(P)n^2\Big(\sum_{\substack{d|n \\ d>2}}\frac{1}{d^2}\Big)+2n\log n+(K+2C_E)n\\&
\leq 4\hat{h}(P)n^2\sum_{d>2}\frac{1}{d^2}+(K+2C_E)n+2n\log n\\&
= 4\hat{h}(P)n^2(\zeta(2)-1-\frac{1}{4})+(K+2C_E)n+2n \log n\\&
\leq 1.6n^2\hat{h}(P)+(K+2C_E)n+2n\log n.
\end{align*}

	Now, we have to deal with the absolute values in $S$. Using Proposition \ref{Siegel} with $\eps=(100\#S)^{-1}$ we have
\begin{align*}
\frac{1}{[K:\Q]}\sum_{\nu \in S}n_\nu h_\nu(nP+Q)&\leq \frac{1}{[K:\Q]}\sum_{\nu\in S}n_\nu\eps h(nP+Q)\\&\leq \frac{h(nP+Q)}{100}\\&\leq \frac{2h(nP)+C_Q}{100}\\&\leq 0.04 n^2\hat{h}(P)+\frac{C_Q+2C_E}{100},
\end{align*}
for $n>\max_{\nu\in S}\{N(\nu,\eps)\}$. 
Finally, using Lemma \ref{quad},
\begin{align*}
2n^2\hat{h}(P)&\leq 2(\hat{h}(nP+Q)+nC(P,Q))\\&\leq h(nP+Q)+2nC(P,Q)+C_E\\&=\frac{\sum_{\nu \in S}n_\nu h_\nu(nP+Q)+\sum_{\nu \notin S}n_\nu h_\nu(nP+Q)}{[K:\Q]}+2nC(P,Q)+C_E\\&\leq 1.64n^2\hat{h}(P)+Jn\log n
\end{align*}
with 
\[
J:=2+2K+4C_E+2C(P,Q)
\]
and then
\[
0.36n^2\leq \frac{Jn\log n}{\hat{h}(P)}
\]
Taking
\[
C:=\max\{2,8\frac{J^2}{\hat{h}(P)^2},\max_{\nu\in S}\{N(\nu,\eps)\}\}
\]
we have that, for every $n>C$, the inequality does not hold and then $D_n$ has a primitive divisor.
	
\end{proof}

\begin{remark}
	The constant $C$ is not effective since Siegel's Theorem it is not.
\end{remark}

\section{An elliptic analogue of Artin's conjecture}\label{secteo2}
Let $E$ be an elliptic curve defined over $\Q$ generated by
\[
y^2+a_1xy+a_3y=x^3+a_2x^2+a_4x+a_6.
\]
Let $S_1$ be the finite set of primes such that $E$ is not in minimal form in $\Q_p$ and such that $E$ has bad reduction. Take $P$ a non-torsion point and $Q$ in $E(\Q)$. We want to prove Theorem \ref{Teo2}.
\begin{proof}[Proof of Theorem \ref{Teo2}]
If $Q$ is in the orbit of $P$ in $E(\Q)$, then the theorem is trivial. So, we can suppose that $nP\neq Q$ for every $n$.
As defined in the introduction, let
\[
N_{P,Q}(x):=\{p\leq x \text{ such that }  p\nmid\Delta \text{  and  } Q \in \langle P \rangle\ \text{ in } E(\F_p)\}.
\]
If $Q$ is in the orbit of $P$ modulo $p$, then there exists $n$ such that \[nP\equiv Q\mod{p}.\] Define \[x(nP+Q)=\frac{C_n}{D_n}\text{   with   } C_n,D_n\in \Z \text{  and  }(C_n,D_n)=1.\]
So, $Q$ is in the orbit of $P$ modulo $p$ if and only if $p$ divides $D_n$ for some $n$. If $p$ divides $D_n$ and it is not in $S_1$, so $\ord_p(D_n)$ is even and then 
\[
\log p^2\leq \log \abs{D_n}\leq h(nP+Q)\leq 2h(nP)+C_Q\leq 4n^2\hat{h}(P)+C_Q+2C_E.
\]
Now, define \[
N=\sqrt{\frac{2\log x-C_Q-2C_E}{4\hat{h}(P)}}.
\]
Thus, if $p$ divides $D_n$ for $1\leq n\leq N$, then
\[
\log p\leq 2n^2\hat{h}(P)+\frac{C_Q}{2}+C_E\leq 2N^2\hat{h}(P)+\frac{C_Q}{2}+C_E= \log x.
\]
So, if $p$ divides $D_n$ for some $1\leq n\leq N$, then $Q$ is in the orbit of $P$ modulo $p$ and $p\leq x$. Thus, we conclude that
\[
\{p\mid D_n \text{ such that } p\notin S_1 \text{  and  } 1\leq n\leq N\}\subseteq N_{P,Q}(x).
\]
Thus,
\[
\omega_\Q(\prod_{n=1}^{N}D_n)\leq \#N_{P,Q}(x)+\#S_1
\]
where $\omega_\Q$ is the function that counts the number of distinct prime divisors of an integer.
In \cite[Theorem 1.1.]{shparlinski}, it was proved that
\[\omega_K(\prod_{n=M}^{N+M}D_n)\gg N\]
and then 
\[
 \#N_{P,Q}(x)\gg \omega_\Q(\prod_{n=1}^{N}D_n)\gg N\gg \sqrt{\log x}.
\]

\end{proof}

\begin{corollary}
	If $P$ is a non-torsion point and $Q$ is a torsion point of prime order, then
	\[
	\lim_{x\to \infty} \frac{\#N_{P,Q}(x)}{\sqrt{\log x}}\geq \frac{1}{\sqrt{2\hat{h}(P)}}.
	\]
\end{corollary}

\begin{proof}
	Thanks to Theorem \ref{Teo1}, we know that $D_n$ has a primitive divisor for every $n\geq C$, with $C$ a constant depending on $E, P$ and $Q$. So,
	\[
	\omega_\Q(\prod_{n=1}^N D_n)\geq N-C
	\]
	and therefore
	\[
	\#N_{P,Q}(x)\geq N-C-\#S_1=\sqrt{\frac{\log x-C_Q/2-C_E}{2\hat{h}(P)}}-C-\#S_1.
	\]
	Hence,
	\[
		\lim_{x\to \infty} \frac{\#N_{P,Q}(x)}{\sqrt{\log x}}\geq\lim_{x\to \infty} \frac{\sqrt{\frac{\log x-C_Q/2-C_E}{2\hat{h}(P)}}-C-\#S_1}{\sqrt{\log x}}= \frac{1}{\sqrt{2\hat{h}(P)}}.
	\]
\end{proof}
\normalsize
\baselineskip=17pt
\bibliographystyle{plain}
\bibliography{biblio}

UNIVERSIT\'A DI PISA, DIPARTIMENTO DI MATEMATICA, LARGO BRUNO PONTECORVO 5, PISA, ITALY\\
\textit{E-mail address}: verzobio@student.dm.unipi.it

\end{document}